\documentclass{amsart}
\usepackage{mathrsfs,amssymb}
\usepackage[T1]{fontenc}
\usepackage{xcolor}

\newtheorem{Thm}{Theorem}
\newtheorem{Prop}[Thm]{Proposition}
\newtheorem{Lemma}[Thm]{Lemma}
\theoremstyle{definition}
\newtheorem{Def}[Thm]{Definition}

\theoremstyle{remark}

\title[Grading on $\mathrm{UT}_n^{(-)}$ revisited]{Gradings on the algebra of triangular matrices as a Lie algebra: revisited}
\author{Plamen Koshlukov}
\address{Department of Mathematics, State University of Campinas, Campinas, Brazil}
\thanks{P. Koshlukov was partially supported by FAPESP grant No.~2018/23690-6 and by CNPq grant No.~302238/2019-0}
\email{plamen@unicamp.br}
\author{Felipe Yukihide Yasumura}
\address{Instituto de Matem\'atica e Estat\'istica, Universidade de S\~ao Paulo, S\~ao Paulo, Brazil}
\email{fyyasumura@ime.usp.br}

\begin{document}
\begin{abstract}
We investigate the group gradings on the algebra of upper triangular matrices over an arbitrary field, viewed as a Lie algebra. These results were obtained a few years early by the same authors. We provide streamlined proofs, and present a complete classification of isomorphism classes of the gradings. We also provide a classification of the practical isomorphism classes of the gradings, which is a better alternative way to consider these gradings up to being essentially the same object. Finally, we investigate in details the case where the characteristic of the base field is $2$, a topic that was neglected in previous works.
\end{abstract}
\maketitle


\section{Introduction}
The classification of all possible group gradings on a given algebra is a very important and interesting problem. One can trace its origins back to the paper by Wall \cite{wall} where he classified all finite dimensional associative algebras which are simple as $C_2$-graded algebras. Here and in what follows, $C_2$ stands for the cyclic group of order 2. Later on Patera and Zassenhaus \cite{PZ} started a systematic study on the subject with a special focus on simple Lie algebras. An essentially complete classification of all group gradings on simple Lie algebras over algebraically closed fields is known as a result of the efforts of many authors, see \cite{EK2013} for the state-of-art. There were many advances in the context of other simple algebras as well. We refer to the monograph \cite[Chapter 2]{EK2013} for the classification of the gradings on associative matrix algebras and for the corresponding bibliographic references. It should be noted that there have appeared studies concerning the classification of group gradings in the context of non-simple algebras, see for instance \cite{Bahturin}. Group gradings on non-simple algebras have compelling connections to PI-algebras as well. We recall that block-triangular matrices play an important role in PI-theory, see, for instance, \cite{GZbook}. Their related graded PI-properties have been a subject of several recent studies, see \cite{CM2015,DM2016,DS2014,RD2019,DPS2020,DSS2019,DP2020}, to cite only a few examples. The group gradings on the upper block-triangular matrices were computed in \cite{VZ2012,BFD2018,y2018}. Here we recall the sophisticated theory developed by A. Kemer in the eighties. It connects varieties of associative algebras over a field of characteristic 0, to the Grassmann envelopes of finite-dimensional $C_2$-graded algebras (see, for instance, \cite{Kemerbook,AKBK}). It is also worth mentioning a series of other works relating gradings and PI-properties. We cite here the papers \cite{kz, aljadeff_darell, shz, by}, and the references therein.

The algebra of upper triangular matrices over an arbitrary field can be viewed as a Lie algebra. It is a special case of block-triangular matrices, and it will be the main object of the present paper. In the associative setting, the papers  \cite{VZ2007,DKV2004} completed the classification of isomorphism classes of groups gradings on the upper triangular matrix algebras. In the non-associative setting, group gradings on the same algebra viewed as a Lie algebra were obtained in \cite{KY2017}. As a Jordan algebra, the classification of group gradings for $n=2$ was done in \cite{KM2012}, and then, generalized for arbitrary $n$ in \cite{KY2017a}. These latter papers concerning the non-associative setting excluded the possibility of characteristic 2 for the base field.

Meanwhile, several authors were interested in the study of (graded) PI-related properties of the algebra of $2\times2$ upper triangular matrices viewed as a Jordan algebra \cite{CenM2017,CenMS2019,CM2014}. Also, recent works were dedicated to investigating the same algebra over finite fields \cite{DimasMateus}, and over fields of characteristic 2 \cite{ManuPedro}. All this leads us to think that it would be interesting to obtain a complete classification of group gradings on the algebra of $n\times n$ upper triangular matrices considered as a Lie algebra over arbitrary fields (including characteristic 2). This is our main goal in this paper. As a by-product, we obtain shorter and streamlined proofs of the classification of the isomorphism classes on the same algebra over any field, thus recovering the results of \cite{KY2017}.

Let $\mathrm{UT}^{(-)}_n=\mathrm{UT}_n^{(-)}(\mathbb{F})$ denote the vector space of all upper triangular matrices of order $n$ with entries in an arbitrary field $\mathbb{F}$, endowed with the bracket $[a,b]=ab-ba$. This binary operation evidently makes $\mathrm{UT}^{(-)}_n$ a Lie algebra. The purpose of this paper is to obtain the classification of isomorphism classes of group gradings on it. We also shall suggest an alternative approach to the subject, namely, we classify the so-called \textsl{practical isomorphism classes} of group gradings.

\section{Preliminaries}
\subsection{Graded algebras}
Let $\mathcal{A}$ be an arbitrary algebra over a field $\mathbb{F}$, and let $G$ be any group. We use multiplicative notation for $G$, and denote its identity by $1$. We say that $\mathcal{A}$ is \textsl{$G$-graded} if $\mathcal{A}$ is endowed with a fixed vector space decomposition,
\[
\Gamma:\mathcal{A}=\bigoplus_{g\in G}\mathcal{A}_g,
\]
where some of the $\mathcal{A}_g$ may be $0$, and such that $\mathcal{A}_g\mathcal{A}_h\subseteq \mathcal{A}_{gh}$, for all $g$, $h\in G$. The subspace $\mathcal{A}_g$ is called the \textsl{homogeneous component of degree $g$}, and the non-zero elements $x\in\mathcal{A}_g$ are said to be homogeneous of degree $g$. We write $\deg x=g$ for these elements. The \textsl{support} of $\mathcal{A}$ (or of $\Gamma$) is the set $\mathrm{Supp}\,\mathcal{A}=\{g\in G\mid\mathcal{A}_g\ne 0\}$.

A subspace $\mathcal{I}\subset\mathcal{A}$ is called \textsl{graded} (or homogeneous) if $\mathcal{I}=\bigoplus_{g\in G}\mathcal{I}\cap\mathcal{A}_g$. If $\mathcal{I}$ is a \textsl{graded ideal} (that is, it is simultaneously an ideal and a graded subspace), then the quotient algebra $\mathcal{A}/\mathcal{I}$ inherits in a natural way the $G$-grading.

If $\mathcal{A}$ is an associative or a Lie algebra, then a \textsl{graded $\mathcal{A}$-module} is an $\mathcal{A}$-module $\mathcal{V}$ with a fixed vector space decomposition $\mathcal{V}=\bigoplus_{g\in G}\mathcal{V}_g$ such that $\mathcal{A}_g\mathcal{V}_h\subseteq\mathcal{V}_{gh}$, for all $g$, $h\in G$.

Let $\mathcal{B}=\bigoplus_{g\in G}\mathcal{B}_g$ be another $G$-graded algebra. A map $f:\mathcal{A}\to\mathcal{B}$ is called a \textsl{homomorphism of $G$-graded algebras} if $f$ is a homomorphism of algebras and $f(\mathcal{A}_g)\subset\mathcal{B}_g$, for all $g\in G$. If, moreover, $f$ is an isomorphism, we call $f$ a \textsl{$G$-graded isomorphism} (or an isomorphism of graded algebras), and we say that $\mathcal{A}$ and $\mathcal{B}$ are \textsl{$G$-graded isomorphic} (or isomorphic as graded algebras). Two $G$-gradings, $\Gamma$ and $\Gamma'$, on the same algebra $\mathcal{A}$ are \textsl{isomorphic} if $(\mathcal{A},\Gamma)$ and $(\mathcal{A},\Gamma')$ are isomorphic as graded algebras.

Now, if $f:\mathcal{A}\to\mathcal{A}$ is an algebra automorphism, then $f$ induces a new $G$-grading on $\mathcal{A}$ via $\mathcal{A}=\bigoplus_{g\in G}\mathcal{A}'_g$, where $\mathcal{A}'_g=f(\mathcal{A}_g)$. Note that $f$ will be an isomorphism between these two gradings on $\mathcal{A}$.

\subsection{Upper triangular matrices}
Let $\mathbb{F}$ be an arbitrary field, and $\mathrm{UT}^{(-)}_n=\mathrm{UT}_n(\mathbb{F})^{(-)}$ be the vector space of the upper triangular matrices endowed with the bracket operation
\[
[a,b]=ab-ba.
\]
We use the convention that the bracket is left-normed, that is, we inductively define
\[
[a_1,a_2,\ldots,a_m]:=[[a_1,a_2,\ldots,a_{m-1}],a_m],\quad m>2.
\]
Note that $\mathrm{UT}^{(-)}_n$ is a Lie algebra. We denote by $e_{ij}$ the matrix units, that is, the matrix having entry $1$ in the $(i,j)$ entry and $0$ elsewhere. We use $\mathrm{I}$ to denote the identity matrix, so $\mathrm{I}=\sum_{i=1}^ne_{ii}$. The centre of the algebra is $\mathbb{F}\mathrm{I}$.

We give several important examples of group gradings on $\mathrm{UT}^{(-)}_n$. In fact, we shall prove later on that these four classes (indeed, essentially two classes) of examples exhaust the list of isomorphism classes on it.

\subsubsection{Elementary grading}
Let $\textbf{g}=(g_1,\ldots,g_n)\in G^n$ be a sequence, and for simplicity, suppose that the $g_i$ are pairwise commuting elements. Then we can define a $G$-grading on $\mathrm{UT}^{(-)}_n$ by $\mathrm{UT}^{(-)}_n=\bigoplus_{g\in G}\mathcal{U}_g$, where
\[
\mathcal{U}_g=\mathrm{Span}\{e_{ij}\mid g_ig_j^{-1}=g\}.
\]
It is an easy exercise to verify that this decomposition gives a $G$-grading on $\mathrm{UT}^{(-)}_n$. This kind of grading is called \textsl{elementary}. It is worth mentioning that some authors reserve the name \textsl{elementary} to refer to a grading that is isomorphic to an elementary one, as constructed above.

The following assertions are equivalent ways to define an elementary grading:
\begin{enumerate}
\item a $G$-grading on $\mathrm{UT}_n^{(-)}$ is elementary if and only if all matrix units are homogeneous in the grading,
\item a $G$-grading on $\mathrm{UT}_n^{(-)}$ is elementary if and only if all matrix units $e_{11}$, \dots, $e_{nn}$ are homogeneous in the grading. 
\end{enumerate}
Also, if we know that it is given an elementary grading on $\mathrm{UT}_n^{(-)}$, then the knowledge of the sequence
\[
\eta=(\deg e_{12},\deg e_{23},\ldots,\deg e_{n-1,n})
\]
gives enough information to recover all the grading. Indeed, if $i<j$ then each
\[
e_{ij}=[e_{i,i+1}, e_{i+1,i+2}, \ldots,e_{j-1,j}]
\]
is homogeneous, and its degree is the product of the degrees of $e_{i,i+1}$, \dots,  $e_{j-1,j}$. Moreover, since each $e_{ii}$ is homogeneous and $e_{i,i+1}=[e_{ii}, e_{i,i+1}]$, we obtain $\deg e_{ii}=1$. Thus, the sequence completely defines the elementary grading. Conversely, given a sequence of $\eta\in G^{n-1}$ of $n-1$ commuting elements, we may construct an elementary grading on $\mathrm{UT}_n^{(-)}$ by declaring $\deg e_{i,i+1}=g_i$, and then extending it, as above, to the remaining $e_{ij}$.

It will be convenient for us to say that such an elementary grading is defined by the sequence $\eta$, and we shall denote the grading by $\Gamma_1(\eta)$.

\subsubsection{Practical elementary grading\label{example}} The next construction is an elementary grading, but we ``change'' the degree of the identity matrix.

Let $\eta=(g_1,\ldots,g_{n-1})\in G^{n-1}$ be a sequence of pairwise commuting elements, and let $t\in G$. Let $\mathscr{D}\subseteq\mathrm{UT}_n^{(-)}$ be the set of diagonal matrices, and let $V$ be any subspace such that $\mathscr{D}=\mathbb{F}\mathrm{I}\oplus V$. Then, we obtain a $G$-grading on $\mathrm{UT}_n^{(-)}$ by declaring:
\begin{enumerate}
\item each $e_{i,i+1}$ is homogeneous of degree $g_i$,
\item the identity matrix $\mathrm{I}$ is homogeneous and $\deg\mathrm{I}=t$,
\item every non-zero element in $V$ is homogeneous of degree $1$.
\end{enumerate}
In this construction, similarly as mentioned above, each $e_{ij}$, for $i<j$, will be a homogeneous element of degree $g_ig_{i+1}\cdots g_{j-1}$. This decomposition works like an elementary grading, with the exception of the identity matrix. However, the identity annihilates any product, so its degree does not interfere with the grading. This grading will be denoted by the pair $\Gamma_1(t,\eta)$. Formally, we would need to consider the subspace $V$; however, as we will prove later, the use of a different subspace would produce an isomorphic graded algebra. So $\Gamma_1(t,\eta)$ denotes an isomorphism class of gradings on $\mathrm{UT}_n^{(-)}$.

Note that every elementary grading is a special case of a practical elementary grading. In this case, $\deg\mathrm{I}=1$. So $\Gamma_1(t,\eta)$ represents the elementary grading defined by $\eta$ if and only if $t=1$.

\subsubsection{Type 2 grading}
In this subsection, we suppose $\mathrm{char}\,\mathbb{F}\ne2$. Let $g\in G$ be of order 2, that is $g\ne g^2=1$, and let $\eta=(g_1,\ldots,g_{n-1})\in G^{n-1}$ be a symmetric sequence. This means that $g_i=g_{n-i+1}$, for each $i$. Assume further that $g$, $g_1$, \dots, $g_{n-1}$ commute pairwise. For every  $i$, $j\in\mathbb{N}$, with $1\le i\le j\le n$, let
\[
X_{ij}^\pm:=e_{ij}\pm e_{n-j+1,n-i+1}.
\]
\begin{Def}\label{type2grading}
Let $\mathrm{char}\,\mathbb{F}\ne2$, and let $g$ and $\eta$ be as above. A \textsl{Type 2} grading is a $G$-grading on $\mathrm{UT}_n^{(-)}$ where each $X_{ij}^\pm$ is homogeneous and
\begin{align*}
\deg X_{ii}^-=1,&\quad \deg X_{ii}^+=g,\\%
\deg X_{i,i+1}^-=g_i,&\quad \deg X_{i,i+1}^+=gg_i.
\end{align*}
We denote such a grading by $\Gamma_2(g,\eta)$.
\end{Def}

A particular case of a type 2 grading may be constructed from the canonical involution of $\mathrm{UT}_n$ in a natural way, as follows. Let $\tau\colon \mathrm{UT}_n\to\mathrm{UT}_n$ be such that $\tau(e_{ij})=e_{n-j+1,n-i+1}$, for all $1\le i\le j\le n$. Then $\tau$ is an involution, that is $\tau^2=1$, and $\tau(ab)=\tau(b)\tau(a)$, for all $a$, $b\in\mathrm{UT}_n$. In fact $\tau$ represents the reflection along the second diagonal, from top right to bottom left in the matrices; this is an involution on the upper triangular matrices. Since $\eta$ is chosen to be symmetric, then $\tau$ is a \textsl{graded involution} for $\Gamma_1(\eta)$, that is, $\deg \tau(x)=\deg x$, for each homogeneous $x\in(\mathrm{UT}_n,\Gamma_1(\eta))$. Since $\tau$ is an involution, we obtain a vector space decomposition
\[
\mathrm{UT}_n=\mathcal{K}\oplus\mathcal{H},
\]
where $\mathcal{K}$ is the set of the skew-symmetric elements of $\mathrm{UT}_n$ with respect to $\tau$, and $\mathcal{H}$ is the set of all symmetric elements. Note that such a decomposition provides a $C_2$-grading on $\mathrm{UT}_n^{(-)}$, where $C_2$ is the group of order $2$. Pay attention that it is not a $C_2$-grading on the associative algebra $\mathrm{UT}_n$. Since $\tau$ is a graded involution, both subspaces $\mathcal{K}$ and $\mathcal{H}$ are $G$-graded. So, considering the intersection of each homogeneous subspace defining the grading $\Gamma_1(\eta)$ with the spaces $\mathcal{K}$ and $\mathcal{H}$, we obtain a $(C_2\times G)$-grading on $\mathrm{UT}_n^{(-)}$. One can check that this construction is a particular case of a type 2 grading. However, given an element $g\in G$ of order $2$, the sequence
\[
1\to\langle g\rangle\to G\to G/\langle g\rangle\to1
\]
is not always split. Thus Definition \ref{type2grading} is slightly more general, even though it is more tedious to check its consistence for a grading.

Now let $\epsilon_1$, \dots, $\epsilon_q\in\mathbb{F}^\times$, where $q=\lfloor\frac{n-1}2\rfloor$ and $\mathbb{F}^\times$ is the multiplicative group of invertible elements in $\mathbb{F}$. Then we may construct a grading on $\mathrm{UT}_n^{(-)}$ by requiring the elements
\begin{equation}\label{type2grading2}
\begin{split}
X_{ii}^\pm,\quad i=1,2,\ldots,\left\lceil\frac{n-1}2\right\rceil,\\%
e_{i,i+1}\pm\epsilon_ie_{n-i,n-i+1},\quad i=1,\ldots,q,
\end{split}
\end{equation}
to be homogeneous. To prove that this grading is consistent, we shall prove that it is isomorphic to one as constructed in Definition \ref{type2grading}. So let $M$ be the invertible diagonal $n\times n$ matrix given by
\[
M=\mathrm{diag}(1,\ldots,1,\epsilon_1\cdots\epsilon_q,\epsilon_2\cdots\epsilon_q,\ldots,\epsilon_q,1).
\]
Note that
\[
MX_{i,i+1}^\pm M^{-1}=e_{i,i+1}\pm\epsilon_i e_{n-i,n-i+1},\quad i=1,2,\ldots,q.
\]
Thus \eqref{type2grading2} defines a grading, and it is isomorphic to a type 2 grading. Moreover, as another consequence, taking every $\epsilon_i$ to be either $1$ or $-1$, we see that $\Gamma_2(g,\eta)\cong\Gamma_2(g,\eta')$, whenever $\eta'=(g_1',\ldots,g_{n-1}')$ is such that some of the $g_i'$ differs from $g_i$ by $g$. More precisely, denote $\eta\sim_g\eta'$ if:
\begin{enumerate}
\item for each $1\le i\le\lfloor\frac{n-1}2\rfloor$, $g_i'$ equals either $g_i$ or $gg_i$,
\item if $n$ is even, then $g_{n/2}'=g_{n/2}$.
\end{enumerate}
Since $g^2=1$, then $\sim_g$ is an equivalence relation. Then the following isomorphism holds.
\begin{Lemma}\label{iso_type2_grading}
Let $g\in G$, with $g^2=1$, $g\ne 1$, and let $\eta$, $\eta\in G^{n-1}$ be symmetric sequences. If $\eta\sim_g\eta'$ then $\Gamma_2(g,\eta)\cong\Gamma_2(g,\eta')$.\qed
\end{Lemma}

\subsubsection{Practical Type 2 grading}
The construction of a practical type 2 grading is similar to the passage from elementary to practical elementary one. It is a Type 2 grading where we change the degree of the identity matrix. Let $\mathrm{char}\,\mathbb{F}\ne2$ and $\Gamma_2(g,\eta)$ be a Type 2 grading, as above, and let $\mathrm{UT}_n^{(-)}=\bigoplus_{h\in G}\mathcal{U}_h$ be the respective decomposition. Let $t\in G$ be any element. The identity matrix $\mathrm{I}$ belongs to the homogeneous component of degree $g$, as in the construction of Type 2 grading. So, let $\mathcal{V}\subseteq\mathcal{U}_g$ be such that $\mathcal{U}_g=\mathcal{V}\oplus\mathrm{Span}\,\{\mathrm{I}\}$. Then we obtain a new $G$-grading on $\mathrm{UT}_n^{(-)}$ if we put $\mathcal{U}_h'=\mathcal{U}_h$, for each $h\in G\setminus\{g,t\}$. Also, we let $\mathcal{U}_t'=\mathcal{U}_t+\mathbb{F}\mathrm{I}$, and, if $g\ne t$, then $\mathcal{U}_g'=\mathcal{V}$. This grading will be denoted by $\Gamma_2(t,g,\eta)$, and we shall call it \textsl{practical type 2} grading. Clearly a practical type 2 grading $\Gamma_2(t,g,\eta)$ is the type 2 grading $\Gamma_2(g,\eta)$ if and only if $t=g$.

\subsection{Practical graded isomorphism}
\label{sect23}
We shall present some methods that appeared in \cite{KY2017}, and were later improved in \cite[Section 7]{KY2019}, which enable us to deal with gradings on non-simple Lie algebras.

Let $\mathcal{L}$ be a $G$-graded algebra and $\mathcal{M}$ a $G$-graded $\mathcal{L}$-module. Given a subset $S\subseteq\mathcal{M}$ and a graded subalgebra $\mathcal{H}\subseteq\mathcal{L}$, we denote
\[
\mathrm{C}_\mathcal{H}\,S=\{b\in\mathcal{H}\mid bm=0, \textrm{ for every } m\in S\},
\]
the centralizer of $S$ in $\mathcal{H}$.
Note that $\mathfrak{z}(\mathcal{L})=\mathrm{C}_\mathcal{L}\,\mathcal{L}$ where $\mathfrak{z}(\mathcal{L})$ stands for the centre of $\mathcal{L}$. The following result is well-known, and will be used repeatedly in what follows. We include its proof for completeness.
\begin{Prop}
If $\mathrm{Span}\,S$ is a graded subspace, then so is $\mathrm{C}_\mathcal{H}\,S$.
\end{Prop}
\begin{proof}
Let $x\in\mathrm{C}_\mathcal{H}\,S$. Then $x\in\mathcal{H}$, and since $\mathcal{H}$ is graded, we may write $x=\sum_{g\in G}x_g$, where each $x_g\in\mathcal{L}_g\cap\mathcal{H}$. Clearly $\mathrm{C}_\mathcal{H}\,S=\mathrm{C}_\mathcal{H}(\mathrm{Span}\,S)$, so let $s\in\mathrm{Span}\,S$ be a homogeneous element. Then
$$
0=xs=\sum_{g\in G}x_gs.
$$
Since each $x_gs$ is a homogeneous element, we necessarily have $x_gs=0$ for all $g\in G$. Since $s$ is an arbitrary homogeneous element and $\mathrm{Span}\,S$ is a graded subspace, then $x_g\in\mathrm{C}_\mathcal{H}(\mathrm{Span}\,S)$ for all $g\in G$. Thus, $\mathrm{C}_\mathcal{H}\,S$ is a graded subspace.
\end{proof}

As an immediate consequence, since $\mathfrak{z}\left(\mathrm{UT}_n^{(-)}\right)=\mathbb{F}\mathrm{I}$, we obtain that $\mathrm{I}$ is always a homogeneous element. Since $\mathrm{I}$ annihilates any product, it may have its degree changed without essentially modifying a grading on $\mathrm{UT}_n^{(-)}$. In order to control the centre of the algebra, we recall the following notion.

\begin{Def}[{\cite[Definition 7]{KY2017}}]\label{practical_iso}
Let $\mathcal{L}_1$ and $\mathcal{L}_2$ be $G$-graded algebras. $\mathcal{L}_1$ and $\mathcal{L}_2$ are said to be \textsl{practically $G$-graded isomorphic} if there exists an isomorphism of (ungraded) algebras $\psi:\mathcal{L}_1\to\mathcal{L}_2$ that induces a $G$-graded isomorphism $\mathcal{L}_1/\mathfrak{z}(\mathcal{L}_1)\to \mathcal{L}_2/\mathfrak{z}(\mathcal{L}_2)$.
\end{Def}

The definition says that a practical elementary grading $\Gamma_1(t,\eta)$ is practically graded isomorphic to the elementary grading $\Gamma_1(\eta)$. Conversely, a $G$-grading $\Gamma$ on $\mathrm{UT}_n^{(-)}$ which is practically $G$-graded isomorphic to an elementary grading defined by $\eta$ is necessarily the practical elementary grading $\Gamma'=\Gamma_1(\deg\mathrm{I},\eta)$. Indeed, the identity map $(\mathrm{UT}_n^{(-)},\Gamma)\to(\mathrm{UT}_n^{(-)},\Gamma')$ satisfies Definition \ref{practical_iso}. A similar discussion holds in the context of type 2 gradings if $n>2$. If $n=2$, then the practical elementary grading $\Gamma_1(t,\eta)$ is always practically graded isomorphic to a type 2 grading $\Gamma_2(g,\eta)$. They are graded isomorphic if and only if $t\ne t^2=1$, and $g=t$. We record this assertion for future reference.
\begin{Lemma}\label{fact}
Let $\Gamma$ be a $G$-grading on $\mathrm{UT}_n^{(-)}$.
\begin{enumerate}
\item $\Gamma$ is practically $G$-graded isomorphic to the elementary grading $\Gamma_1(\eta)$ if and only if it is $G$-graded isomorphic to the practical elementary grading $\Gamma_1(\deg\mathrm{I},\eta)$.
\item If $n>2$, then $\Gamma$ is practically $G$-graded isomorphic to a type 2 grading $\Gamma_2(g,\eta)$ if and only if it is $G$-graded isomorphic to the practical type 2 grading $\Gamma_2(\deg\mathrm{I},g,\eta)$.
\end{enumerate}
Additionally, if $n=2$ and $\mathrm{char}\,\mathbb{F}\ne2$, then the practical elementary grading $\Gamma_1(t,\eta)$ is practically graded isomorphic to a type 2 grading $\Gamma_2(g,\eta)$. Moreover, they are graded isomorphic if and only if $t=g$ (that is, one needs that $t\ne t^2=1$).\qed
\end{Lemma}

In order to compare two practical elementary gradings or two type 2 gradings defined by the same parameters, we recall the following fact.

\begin{Lemma}[{\cite[Theorem 27]{KY2019}}]\label{kythm}
Let $\mathcal{L}_1$ and $\mathcal{L}_2$ be $G$-graded Lie algebras, and assume that there exists an isomorphism of (ungraded) algebras $\psi:\mathcal{L}_1\to\mathcal{L}_2$ such that both the induced map $\mathcal{L}_1/\mathfrak{z}(\mathcal{L}_1)\to \mathcal{L}_2/\mathfrak{z}(\mathcal{L}_2)$ and the restriction $\mathfrak{z}(\mathcal{L}_1)\to\mathfrak{z}(\mathcal{L}_2)$ are $G$-graded isomorphisms. Then $\mathcal{L}_1$ and $\mathcal{L}_2$ are isomorphic as $G$-graded algebras.\qed
\end{Lemma}

As an immediate consequence, we obtain that two practical elementary gradings on $\mathrm{UT}_n^{(-)}$, defined by the same pair $(t,\eta)$ are isomorphic as $G$-graded algebras. Indeed, one may apply Lemma \ref{kythm} where $\psi$ is chosen to be the identity map. This justifies the reason why the choice of the complementing subspace used in the construction of \ref{example} is irrelevant in terms of isomorphism classes of gradings on $\mathrm{UT}_n^{(-)}$. The same discussion remains valid for practical type 2 gradings.

\subsubsection{Semihomogeneous elements} The following technical definition will also be important for our proofs.

\begin{Def}[{\cite[Definition 3]{KY2019}}]
Let $\mathcal{L}$ be a $G$-graded algebra. We say that $x\in\mathcal{L}$ is \textsl{semihomogeneous} of degree $g$ if there exist $y\in\mathcal{L}_g$ and $z\in\mathfrak{z}(\mathcal{L})$ such that $x=y+z$.
\end{Def}

Notice that every homogeneous element is a particular case of a semihomogeneous element. Also, if $x$ and $y$ are semihomogeneous elements of degree $g$ and $h$, respectively, then their bracket $[x,y]$ is either $0$ or is homogeneous of degree $gh$.

Considering $\mathrm{UT}_n^{(-)}$, we have the following result.

\begin{Lemma}\label{semihomogeneous_elem}
Assume that $\mathrm{UT}_n^{(-)}$ is $G$-graded in such a way that $e_{11},e_{22},\ldots,e_{nn}$ are semihomogeneous elements. Then the grading is practically graded isomorphic to an elementary grading.
\end{Lemma}
\begin{proof}
Given $i<j$, we have that $[\mathrm{UT}_n^{(-)},e_{ii},e_{jj}]=\mathbb{F}e_{ij}$ is a graded subspace. Thus $e_{ij}$ is a homogeneous element. So the sequence $(\deg e_{12},\deg e_{23},\ldots,\deg e_{n-1,n})$ provides the construction of a practical elementary grading. Thus, Lemma \ref{fact} tells us that the original grading is practically $G$-graded isomorphic to an elementary grading.
\end{proof}

A similar statement holds in the context of type 2 gradings (see Lemma \ref{semihomogeneous_type2} below).

\subsubsection{Essential support of a grading} Finally, we introduce the support of the grading, up to the centre of the algebra.
\begin{Def}\label{def_ess_supp}
Let $\mathcal{L}=\bigoplus_{g\in G}\mathcal{L}_g$ be $G$-graded. The \textsl{essential support} of the grading is
\[
\{g\in G\mid\mathcal{L}_g\not\subseteq\mathfrak{z}(\mathcal{L})\}.
\]
\end{Def}
If, for instance, $\mathfrak{z}(\mathcal{L})=0$, then the essential support coincides with the ordinary support of the grading. The support of an elementary grading on $\mathrm{UT}_n^{(-)}$ is commutative, and the support of a practical elementary grading is almost commutative in the following sense.
\begin{Lemma}\label{ess_supp}
Let $\mathrm{UT}_n^{(-)}$ be endowed with a practical elementary grading $\Gamma_1(t,\eta)$. Then the essential support of the grading is commutative, that is, if we construct a grading via \ref{example} using an arbitrary $\eta$, then the grading is consistent if and only if the elements of $\eta$ commute pairwise.
\end{Lemma}
\begin{proof}
Suppose that we are given a grading on $\mathrm{UT}_n^{(-)}$ as constructed in \ref{example} (without requiring the elements of $\eta$ to commute pairwise). It is clear that the essential support of the grading is contained in the subgroup generated by $g_1:=\deg e_{12}$, $g_2:=\deg e_{23}$, \dots, $g_{n-1}:=\deg e_{n-1,n}$. Thus, it is enough to prove that the above elements commute pairwise. Let $1\le i<j<n$. If $j=i+1$, then $0\ne[ e_{i,i+1}, e_{j,j+1}]=[e_{j,j+1}, e_{i,i+1}]$, thus $g_ig_j=g_jg_i$. If $j>i+1$. Then we have
\begin{align*}
0\ne-e_{i,j+1}&=[e_{i,i+1},[e_{j,j+1},[e_{i+1,i+2},e_{i+2,i+3},\ldots,e_{j-1,j}]]]\\%
&=[e_{j,j+1},[e_{i,i+1}, e_{i+1,i+2},\ldots, e_{j-1,j}]].
\end{align*}
Thus $g_ig_j(g_{i+1}\cdots g_{j-1})=g_jg_i(g_{i+1}\cdots g_{j-1})$, so $g_ig_j=g_jg_i$.
\end{proof}

In an analogous fashion, we have a similar lemma concerning the Type 2 grading, which we state without proof. Its proof is straightforward and tedious: it consists of a simple computation with the homogeneous elements on the first diagonal above the main diagonal.
\begin{Lemma}\label{ess_supp2}
If $n>2$, then the essential support of a practical type 2 grading is commutative.\qed
\end{Lemma}
It should be noted that the statement of the previous lemma means the following. If we manage to construct a practical type 2 grading, possibly without requiring commutativity of the elements of the group, then we conclude that the essential support of the grading generates an abelian subgroup.

\subsection{Duality of gradings and actions. Automorphism group of $\mathrm{UT}_n^{(-)}$} 
This section is intended to explain the connection between the theory of   algebras graded by an abelian group, and action of groups. We also collect results from \cite{Dokovic} where the Lie automorphisms of upper triangular matrices were computed. This shall give us an insight of what one can expect from the possible (abelian) group gradings on $\mathrm{UT}_n^{(-)}$.

\subsubsection{Duality between gradings and actions} It is well-known that there is a duality between gradings and actions when the base field is algebraically closed of characteristic zero and the grading group is finite abelian. Indeed, suppose all these conditions are met, we have $\hat G:=\mathrm{Hom}(G,\mathbb{F}^\ast)\cong G$, where $\mathbb{F}^\ast$ is the multiplicative group of the non-zero elements of $\mathbb{F}$. Now consider a $G$-grading $\mathcal{A}=\bigoplus_{g\in G}\mathcal{A}_g$ on a finite-dimensional $\mathbb{F}$-algebra $\mathcal{A}$. Then we define a representation of $\hat G$ in $\mathcal{A}$ where $\chi\in\hat G$ acts via $\chi(a)=\chi(g)a$, for every $a\in\mathcal{A}_g$, and $g\in G$. The property concerning product of homogeneous components is equivalent that such representation is a group homomorphism $\hat G\to\mathrm{Aut}(\mathcal{A})$. Conversely, an action of $\hat G$ on $\mathcal{A}$ is completely reducible where each irreducible component is 1-dimensional. Therefore we obtain an eigenspace decomposition of $\mathcal{A}$, indexed by the elements of $\hat{\hat G}\cong G$. Such a decomposition is readily seem to be a $G$-grading, since $\hat G\to\mathrm{Aut}(\mathcal{A})$.
Thus the knowledge of the ordinary automorphism group of the algebra gives us some insight of what we may expect from its group gradings.

It is interesting to note that the duality works for algebraically closed fields of characteristic zero and finitely generated abelian groups (not necessarily finite). In this case, it is well known that $G\cong G_t\times\mathbb{Z}^m$, where $G_t$ is the torsion subgroup of $G$. So $\hat G\cong\hat{G_t}\times(\mathbb{F}^\ast)^m$ is an algebraic group, usually called a \textsl{quasitorus}. Analogously as in the finite abelian case, a $G$-grading implies the existence of an action by $\hat G$. Conversely, note that $\mathrm{Aut}(\mathcal{A})$ is an algebraic group if $\mathcal{A}$ is finite-dimensional. Consider an action of $\hat G$ on $\mathcal{A}$, that is, a homomorphism of algebraic groups $\hat G\to\mathrm{Aut}(\mathcal{A})$. The quasitori enjoy the property that all of their representations are diagonalizable. Thus we obtain a decomposition of $\mathcal{A}$ labelled by $\mathfrak{X}(\hat G)$, where $\mathfrak{X}(\hat G)$ is the group of homomorphisms of algebraic groups $\hat G\to\mathbb{F}^\ast$. Since $G\cong\mathfrak{X}(\hat G)$, we obtain a $G$-grading. Hence, once again we have the duality. So the duality establishes a link between the theory of abelian group gradings and  methods and tools from the theory of algebraic groups and algebraic geometry.

Finally, going one step further, the duality also works over an arbitrary base field if we consider the automorphism group scheme of the algebra (see, for instance, \cite[Section 1.4]{EK2013}). We shall not give the details here. It is worth mentioning that an affine group scheme corresponds to a cocommutative Hopf algebra, and the action of an affine group scheme corresponds to an action of the Hopf algebra (which is also equivalent to a co-action). In the context of abelian gradings, the corresponding Hopf algebra is the group algebra $\mathbb{F}G$, and the respective action and co-action are easily described. Thus, there is also a connection between group gradings and actions of Hopf algebras.

\subsubsection{Automorphism group of $\mathrm{UT}_n^{(-)}$} The automorphism group of $\mathrm{UT}_n^{(-)}$ is known thanks to the paper \cite{Dokovic}, where the author classified the Lie automorphisms of $\mathrm{UT}_n(\mathcal{R})$, where $\mathcal{R}$ is an associative and commutative unital connected ring.

Let $\omega\colon \mathrm{UT}_n^{(-)}\to\mathrm{UT}_n^{(-)}$ be defined by $\omega(e_{ij})=-e_{n-j+1,n-i+1}$. Then it is easy to check that $\omega$ is an automorphism of $\mathrm{UT}_n^{(-)}$ (note that $\omega$ is obtained by an involution with a minus sign). Define $G_0$ as the set of all inner automorphisms of $\mathrm{UT}_n^{(-)}$, that is, all conjugations by an invertible upper triangular matrix. For each sequence $\textbf{a}=(a_1,\ldots,a_n)\in\mathbb{F}^n$ such that $a_1+a_2+\cdots+a_n\ne1$, let $\varphi_\textbf{a}\colon \mathrm{UT}_n^{(-)}\to\mathrm{UT}_n^{(-)}$ be defined by
\[
\varphi_\textbf{a}(e_{ij})=e_{ij}+\delta_{ij}a_i\mathrm{I}.
\]
Let $G_1=\{\varphi_\textbf{a}\mid\textbf{a}=(a_1,\ldots,a_n)\in\mathbb{F}^n,\,a_1+\cdots+a_n\ne1\}$. Then one can verify that $G_1$ is a subgroup of the automorphism group of $\mathrm{UT}_n^{(-)}$. Moverover, \cite[Propositon 3]{Dokovic} states that each element of $G_0$ commutes with each element of $G_1$ and $G_0\cap G_1=1$. The statements of \cite[Propositions 4 and 5]{Dokovic} tell us that $\omega$ normalizes the subgroup $G_0G_1$. Then \cite[Theorem 6]{Dokovic} reads as
\begin{enumerate}
\item $\mathrm{Aut}(\mathrm{UT}_2^{(-)})=G_0G_1$,
\item if $n>2$, then $\mathrm{Aut}(\mathrm{UT}_n^{(-)})=(G_0G_1)\rtimes\langle\omega\rangle$, where $\rtimes$ stands for the semidirect product of the corresponding groups.
\end{enumerate}
So, we may expect that the gradings on $\mathrm{UT}_n^{(-)}$ are, in a sense, close to elementary ones and to type 2 gradings. However, as we may see, if $\mathrm{char}\,\mathbb{F}=2$, there will be no type 2 grading.

It is also relevant to cite the paper \cite{Cao1997}, where the author computes the Lie automorphisms of $\mathrm{UT}_n(\mathcal{R})$, where $\mathcal{R}$ is an associative and commutative unital ring (not necessarily connected). Thus, there is a description of the automorphism group scheme of $\mathrm{UT}_n^{(-)}$ as well.

\section{Statements of the Main results: Classification of group gradings on $\mathrm{UT}_n$}

Our first result concerns a description of the group gradings on $\mathrm{UT}_n^{(-)}$.
\begin{Thm}\label{thm1}
Let $\mathbb{F}$ be a field and let $G$ be any group. Consider a $G$-grading on $\mathrm{UT}_n^{(-)}(\mathbb{F})$. Then the essential support (see Definition \ref{def_ess_supp}) of the grading is commutative. Furthermore, the following two cases can occur:
\begin{enumerate}
\item Either the grading is practically graded isomorphic to an elementary grading $\Gamma_1(\eta)$. In this case, the grading is graded isomorphic to a practical elementary grading $\Gamma_1(\deg\mathrm{I},\eta)$, or
\item the grading is practically graded isomorphic to a type 2 grading $\Gamma_2(g,\eta)$. In this case, if $n\ge 3$ then the grading is graded isomorphic to a practical type 2 grading $\Gamma_2(\deg\mathrm{I},g,\eta)$. For $n=2$, the grading is graded isomorphic to $\Gamma_2(\deg\mathrm{I},\eta)$ if $\deg\mathrm{I}\ne(\deg\mathrm{I})^2=1$.
\end{enumerate}
Moreover, if $\mathrm{char}\,\mathbb{F}=2$, then (2) never happens. If $n=2$ and $\mathrm{char}\,\mathbb{F}\ne2$ then both (1) and (2) are true.
\end{Thm}

The next result states that two $G$-gradings on $\mathrm{UT}_n^{(-)}$ are practically graded isomorphic if and only if they almost satisfy the same graded polynomial identities. More precisely:
\begin{Thm}\label{thm2}
Let $\mathbb{F}$ be an arbitrary field and let $G$ be any group. Assume that $\mathcal{U}_1$ and $\mathcal{U}_2$ denote the same algebra $\mathrm{UT}^{(-)}_n$ but endowed with two $G$-gradings. Let $\mathcal{Z}_1=\mathfrak{z}(\mathcal{U}_1)$ and $\mathcal{Z}_2=\mathfrak{z}(\mathcal{U}_2)$ be their centres. Then the following assertions are equivalent:
\begin{enumerate}
\item $\mathcal{U}_1$ and $\mathcal{U}_2$ are practically $G$-graded isomorphic,
\item $\mathcal{U}_1/\mathcal{Z}_1$ and $\mathcal{U}_2/\mathcal{Z}_2$ are $G$-graded isomorphic,
\item $\mathcal{U}_1/\mathcal{Z}_1$ and $\mathcal{U}_2/\mathcal{Z}_2$ satisfy the same $G$-graded polynomial identities.
\end{enumerate}
\end{Thm}

Finally, we obtain a classification of isomorphism classes and practical isomorphism classes of group gradings on $\mathrm{UT}_n^{(-)}$. Given sequences $\eta=(g_1,\ldots,g_{n-1})$ and $\eta'=(g_1',\ldots,g_{n-1}')$ in $G^{n-1}$, then $\eta\sim\eta'$ denotes that either $\eta=\eta'$ or $(g_1,\ldots,g_{n-1})= \mathrm{rev}\,(\eta')$ where $\mathrm{rev}\,(\eta')= (g_{n-1}',\ldots,g_1')$ stands for the reverse sequence of $\eta'$. Suppose further that $\eta=\mathrm{rev}\,\eta$ and $\eta'=\mathrm{rev}\,\eta'$. Given $g\in G$ of order $2$, recall that $\eta\sim_g\eta'$ means that $g_i$ is $g_i'$ or $gg_i'$, for each $i=1$, 2,  \dots, $\lfloor\frac{n-1}2\rfloor$, and if $n-1$ is odd, then $g_{n/2}=g_{n/2}'$.
\begin{Thm}\label{thm3}
Let $n>2$, $t$, $t'$, $g$, $g'\in G$, with $g\ne g^2=1=\left(g'\right)^2\ne g'$, and $\eta$, $\eta'\in G^{n-1}$. Then
\begin{enumerate}
\item $\Gamma_1(t,\eta)$ and $\Gamma_1(t',\eta')$ are graded isomorphic if and only if $t=t'$ and $\eta\sim\eta'$,
\item $\Gamma_1(t,\eta)$ and $\Gamma_1(t',\eta')$ are practically graded isomorphic if and only if $\eta\sim\eta'$.
\end{enumerate}
For the remaining statements of the theorem, assume further that $n>2$, $\eta=\mathrm{rev}\,\eta$, and $\eta'=\mathrm{rev}\,\eta'$.
\begin{enumerate}
\setcounter{enumi}{2}
\item $\Gamma_2(t,g,\eta)$ and $\Gamma_2(t',g',\eta')$ are graded isomorphic if and only if $g=g'$, $t=t'$, and $\eta\sim_g\eta'$,
\item $\Gamma_2(t,g,\eta)$ and $\Gamma_2(t',g',\eta')$ are practically graded isomorphic if and only if $g=g'$ and $\eta\sim_g\eta'$.
\end{enumerate}
\end{Thm}

\section{Proofs}
Let $\mathbb{F}$ be an arbitrary field, $G$ any group and let $\Gamma$ be a $G$-grading on $\mathrm{UT}_n^{(-)}$. We remark that $\mathfrak{z}(\mathrm{UT}_n^{(-)})=\mathbb{F}\mathrm{I}$, so the identity matrix is always a homogeneous element. Also,
\[
J:=[\mathrm{UT}_n^{(-)},\mathrm{UT}_n^{(-)}]
\]
is always a graded subspace, and it coincides with the set of the strictly upper triangular matrices. Then $J^m:=[J,\ldots,J]$, $m$ times, is graded as well for each $m\in\mathbb{N}$. Note that $J^m=0$ if $m\ge n$, and $J^{n-1}=\mathbb{F}e_{1n}$. Thus $e_{1n}$ is always a homogeneous element.

We shall organize our argument into several steps, as follows. First we deal with the case $n=2$, and separately with the case $n>2$. The case $n>2$ will  essentially split into two parts: $\mathrm{char}\,\mathbb{F}=2$ and $\mathrm{char}\,\mathbb{F}\ne2$. It is worth mentioning that the case $n=2$ was investigated in the paper \cite{ManuPedro} when $\mathrm{char}\,\mathbb{F}=2$.

\subsection{} We include here the case $n=2$ for the sake of completeness. Moreover this will allow us to state the main result using our language and notation.

\begin{Lemma}\label{ut2}
Every $G$-grading on $\mathrm{UT}^{(-)}_2$ is practically graded isomorphic to an elementary one.
\end{Lemma}
\begin{proof}
Let $x=a_{11}e_{11}+a_{22}e_{22}+a_{12}e_{12}\in \mathrm{UT}_2^{(-)}$ be a homogeneous element such that, together with $\mathrm{I}$ and $e_{12}$, it forms a basis for $\mathrm{UT}_2^{(-)}$. Note that $a_{11}\ne a_{22}$, otherwise the set $\{x,\mathrm{I},e_{12}\}$ would not be linearly independent. So $[x, e_{12}]=(a_{11}-a_{22})e_{12}\ne0$ implies that $\deg x=1$. Moreover, the eigenvalues of $x$, namely $a_{11}$ and $a_{22}$, are distinct, so we may find a conjugation $\psi$ of $\mathrm{UT}_2^{(-)}$ such that $\psi(x)=a_{11}e_{11}+a_{22}e_{22}$. Thus, $\psi$ induces a new $G$-grading on $\mathrm{UT}_2^{(-)}$ in such a way that $a_{1}e_{11}+a_{22}e_{22}$ is homogeneous of degree $1$. This grading is a practical elementary grading by definition. Hence, Lemma \ref{fact} completes the proof.
\end{proof}

\subsection{} Thus, from now on, we assume that $n>2$. Given $x\in\mathrm{UT}_n^{(-)}$ denote by $(x)_{ij}$ the entry $(i,j)$ of $x$.

\begin{Lemma}\label{maindivision}
The vector subspace generated by the image of $\mathrm{Span}\{e_{11},e_{nn}\}$ in $\mathrm{UT}_n^{(-)}/(J+\mathbb{F}\mathrm{I})$ is graded. Moreover, $\deg (e_{11}-e_{nn}+J+\mathbb{F}\mathrm{I})=1$, and also $\deg (e_{11}+e_{nn}+J+\mathbb{F}\mathrm{I})=g$, where $g^2=1$ (it may happen that $g=1$).
\end{Lemma}
\begin{proof}
Let, in the notation of  Section \ref{sect23},
\begin{align*}
\mathcal{S}&=C_{\mathrm{UT}_n^{(-)}}(J^{n-2})\\%
&=\left\{x\in\mathrm{UT}_n^{(-)}\mid(x)_{12}=(x)_{n-1,n}=0,\,(x)_{11}=(x)_{22}=(x)_{n-1,n-1}=(x)_{nn}\right\}.
\end{align*}
Then $\mathcal{V}=\mathcal{S}/J^2$ is a graded $\mathrm{UT}_n^{(-)}$-module. Let
\begin{align*}
\mathcal{T}&=C_{\mathrm{UT}_n^{(-)}}(\mathcal{V})\\%
&=\left\{x\in\mathrm{UT}_n^{(-)}\mid(x)_{22}=\cdots=(x)_{n-1,n-1},\,(x)_{i,i+1}=0,i=1,\ldots,n-1\right\}.
\end{align*}

Thus, $(\mathcal{T}+J)/J=\mathrm{Span}\{e_{11}+J,e_{nn}+J,\mathrm{I}+J\}$. Now, let $\mathcal{A}=(\mathcal{T}+J)/J$. Then $J^{n-1}$ is a graded $\mathcal{A}$-module. Note that
\[
\mathcal{A}':=\mathrm{C}_\mathcal{A}\,J^{n-1}=\mathrm{Span}\{\mathrm{I}+J,e_{11}+e_{nn}+J\},
\]
so $e_{11}+e_{nn}+J$ is homogeneous in $\mathcal{A}$ (up to adding a multiple of $\mathrm{I}+J$). Since $\dim\mathcal{A}=3>\dim\mathcal{A}'$, there must exist a homogeneous element $z\in\mathcal{A}$ such that $[z,e_{1n}]\ne0$. This also implies $\deg z=1$. If $\deg (e_{11}+e_{nn}+J)=1$ then we may consider a linear combination of this element and $z$ to prove the lemma. Thus, we may suppose $\deg (e_{11}+e_{nn}+J)\ne 1$, and we shall prove that $z$ is a (linear combination of $\mathrm{I}$, elements of $J$ and of a) multiple of $e_{11}-e_{nn}+J$.

Now $J^{n-2}/J^{n-1}$ is a graded $\mathcal{A}$-module of $\mathbb{F}$-dimension 2. Let $a\in J^{n-2}$ be a homogeneous element having a non-zero image in $J^{n-2}/J^{n-1}$. Note that
\[
[a+J^{n-1},e_{11}+e_{nn}+J,e_{11}+e_{nn}+J]=a+J^{n-1}.
\]
This shows that $(\deg(e_{11}+e_{nn}+J))^2=1$. Since we are supposing $\deg (e_{11}+e_{nn}+J)\ne1$, then necessarily $(a)_{1,n-1}\ne0$ and $(a)_{2,n}\ne0$. Thus a homogeneous $\mathbb{F}$-basis of $J^{n-2}/J^{n-1}$ consists of two elements of distinct degrees. So, since $\deg z=1$, $[z,a]$ must be a multiple of $a$. Thus $(z)_{1,n-1}=-(z)_{2,n}$. Hence $z$ is the image of a multiple of $e_{11}-e_{nn}$. We proved all the statements of the lemma.
\end{proof}

As we shall see, if $g=1$ then we obtain a practical elementary grading, while if $g\ne1$, then we shall obtain a practical type 2 grading.
\begin{Lemma}\label{division_char2}
If $\mathrm{char}\,\mathbb{F}=2$, then, in the notation of Lemma \ref{maindivision}, $g=1$.
\end{Lemma}
\begin{proof}
We shall use the notation of the beginning of the proof of Lemma \ref{maindivision} (where the argument is characteristic-free).

Let $\mathcal{A}=(\mathcal{T}+J)/J$ and
\[
\mathcal{A}':=\mathrm{C}_\mathcal{A}\,J^{n-1}=\mathrm{Span}\{\mathrm{I}+J,e_{11}+e_{nn}+J\}.
\]
As we proved earlier, there exists a homogeneous element of degree $1$ in $\mathcal{A}$.

Now, $J^{n-2}/J^{n-1}$ is a graded $\mathcal{A}$-module. Thus $J^{n-2}/J^{n-1}$ is a graded $\mathcal{A}'$-module as well. Moreover $\mathrm{C}_{\mathcal{A}'}\left(J^{n-2}/J^{n-1}\right)=\mathrm{Span}\{\mathrm{I}+J\}$. So we may find a homogeneous element $z'\in\mathcal{A}'$ such that $z'\notin\mathrm{C}_{\mathcal{A}'}\left(J^{n-2}/J^{n-1}\right)$. Thus
\[
z'=\mu_1(e_{11}+e_{nn})+\mu_2\mathrm{I}+J,\quad\mu_1\ne0.
\]
Therefore $[z',v]\ne0$ for any $0\ne v\in J^{n-2}/J^{n-1}$. This readily implies $\deg z'=1$. Hence we conclude the proof of the lemma.
\end{proof}

\subsection{} We shall investigate the case when $g=1$ in Lemma \ref{maindivision}. Let
\begin{align*}
I_1&=\{x\in\mathrm{UT}_n^{(-)}\mid(x)_{ij}=0\textrm{ if }i\ne1\}.
\end{align*}
If $g=1$, then Lemma \ref{maindivision} says that $e_{11}+J+\mathbb{F}\mathrm{I}$ is a homogeneous element of degree $1$ in $(\mathcal{T}+J)/(J+\mathbb{F}\mathrm{I})$. This means that there exists $r\in J$ such that $e_{11}+r$ is semihomogeneous of degree $1$.

Let $f\colon \mathrm{UT}_n^{(-)}\to\mathrm{UT}_n^{(-)}$ be defined by $f=\mathrm{ad}(e_{11}+r)^n$, where $\mathrm{ad}(e_{11}+r)$ is the linear transformation defined by $\mathrm{ad}(e_{11}+r)(a)=[e_{11}+r,a]$.

\begin{Lemma}\label{graded_ideal}
The map $f$ is a $G$-graded linear map and $f(\mathrm{UT}_n^{(-)})=I_1\cap J$.
\end{Lemma}
\begin{proof}
Since $e_{11}+r$ is a semihomogeneous element of degree $1$ and the product is bilinear, the map $f$ is a $G$-graded linear map. Now note that $I_1\cap J$ is an ideal. Then an inductive argument shows that
\[
\mathrm{ad}(e_{11}+r)^m(\mathrm{UT}_n^{(-)})\subseteq I_1\cap J+J^m,\quad m>0.
\]
Therefore $f(\mathrm{UT}_n^{(-)})\subseteq I_1\cap J+J^n= I_1\cap J$. Conversely, for every $m>1$, we obtain $f(e_{1m})=e_{1m}+u$, for some $u\in I_1\cap J^{m+1}$. This proves that $I_1\cap J\subseteq f(\mathrm{UT}_n^{(-)})$.
\end{proof}

As a consequence of Lemma \ref{graded_ideal}, $I_1\cap J$ is a graded ideal of $\mathrm{UT}_n^{(-)}$. It follows that $\mathcal{B}:=\mathrm{UT}_n^{(-)}/I_1\cap J$ is a graded algebra, and
\[
C_\mathcal{B}\,\mathcal{B}=\mathrm{Span}\{\mathrm{I}+I_1\cap J,e_{11}+I_1\cap J\}.
\]
So, taking an inverse image, we obtain the proof of the following statement.
\begin{Lemma}
If $g=1$ in Lemma \ref{maindivision}, then the subspace $I_1\oplus\mathbb{F}\mathrm{I}$ is always a graded subspace of $\mathrm{UT}_n^{(-)}$.\qed
\end{Lemma}

Now let $x\in I_1\oplus\mathbb{F}\mathrm{I}$ be a homogeneous element with $(x)_{11}=1$ and $(x)_{nn}\ne1$. Such element exists, because $e_{11}\in I_1\oplus\mathbb{F}\mathrm{I}$. Since $[x,e_{1n}]=((x)_{11}-(x)_{nn})e_{1n}\ne0$, we see that $\deg x=1$. This implies that there exists $r_0\in I_1\cap J$ such that $e_{11}+r_0$ is semihomogeneous of degree $1$. Now the minimal polynomial of $e_{11}+r_0$ (viewed as a matrix) is $X(X-1)$, therefore $e_{11}+r_0$ is diagonalizable 
(indeed, $e_{11}+r_0$ is an idempotent matrix hence diagonalizable). The conjugation that maps $e_{11}+r_0\mapsto e_{11}$ induces a new $G$-grading on $\mathrm{UT}_n^{(-)}$, isomorphic to the original one. Hence we proved the following lemma.
\begin{Lemma}\label{e_homogeneous}
If $g=1$ in Lemma \ref{maindivision}, then there exists an isomorphic $G$-grading on $\mathrm{UT}_n^{(-)}$ where $e_{11}$ is semihomogeneous of degree $1$.\qed
\end{Lemma}
From now on, let us assume that $e_{11}$ is semihomogeneous of degree $1$. Let
\[
\mathcal{C}=(1-\mathrm{ad}\,e_{11})\mathrm{UT}_n^{(-)}:=\left\{x-[e_{11},x]\mid x\in\mathrm{UT}_n^{(-)}\right\}.
\]
Then $\mathcal{C}$ is a graded subalgebra; it coincides with the direct sum 
\[
\{x\in\mathrm{UT}_n^{(-)}\mid(x)_{1j}=0,\textrm{ for every } j=1,\ldots,n\}\oplus\mathbb{F}\mathrm{I}.
\] 
Also, $\mathcal{C}\cong\mathrm{UT}_{n-1}^{(-)}\oplus\mathbb{F}\mathrm{I}$. We prove the following lemma.
\begin{Lemma}\label{existence_semihomogeneous}
If $g=1$ in Lemma \ref{maindivision}, then there exists an isomorphic $G$-grading on $\mathrm{UT}_n^{(-)}$ where $e_{11},e_{22},\ldots,e_{nn}$ are semihomogeneous of degree $1$.
\end{Lemma}
\begin{proof}
Repeating the steps of the proof of Lemma \ref{e_homogeneous}, we see that $e_{22}$ is semihomogeneous of degree $1$ in $\mathcal{C}$. However, $\mathfrak{z}\left(\mathcal{C}\right)=\mathrm{Span}\{e_{11},\mathrm{I}\}$. So, $e_{22}+\lambda_1e_{11}+\lambda_2\mathrm{I}$ is homogeneous of degree $1$, for some $\lambda_1$, $\lambda_2\in\mathbb{F}$. Since $e_{11}$ is semihomogeneous of degree $1$, we get that $e_{22}$ is semihomogeneous of degree $1$ in $\mathrm{UT}_n^{(-)}$. Thus, an inductive argument completes the proof.
\end{proof}

\subsection{} Now we shall investigate the case where $g\ne1$ in Lemma \ref{maindivision}. By Lemma \ref{division_char2}, we shall assume that $\mathrm{char}\,\mathbb{F}\ne2$. We observe that the arguments presented here could also be used in order to re-obtain the results of the case $g=1$ if $\mathrm{char}\,\mathbb{F}\ne2$.

Recall that $(x)_{ij}$ stands for the entry $(ij)$ of the matrix $x$. Let
\begin{align*}
D_1&=\left\{x\in\mathrm{UT}_n\mid(x)_{ij}=0\text{ if $i\ne1$ and $j\ne n$}\right\}.
\end{align*}
\begin{Lemma}
The subspace $D_1\oplus\mathbb{F}\mathrm{I}$ is a graded subspace. Moreover, the images of the elements $e_{11}\pm e_{nn}$ are homogeneous in $\mathrm{UT}_n^{(-)}/((D_1\cap J)\oplus\mathbb{F}\mathrm{I})$. Also
\[
\deg (e_{11}-e_{nn}+(D_1\cap J)\oplus\mathbb{F}\mathrm{I})=\left(\deg e_{11}+e_{nn}+(D_1\cap J)\oplus\mathbb{F}\mathrm{I}\right)^2=1.
\]
\end{Lemma}
\begin{proof}
Note that $D_1\cap J$ is an ideal. By Lemma \ref{maindivision}, the element $e_{11}-e_{nn}+r$ is semihomogeneous for some $r\in J$. An easy induction shows that
\[
\left(\mathrm{ad}(e_{11}-e_{nn}+r)\right)^m(\mathrm{UT}_n^{(-)})\subseteq (D_1\cap J)+J^m,\quad m>0.
\]
Thus $\left(\mathrm{ad}(e_{11}-e_{nn}+r)\right)^n(\mathrm{UT}_n^{(-)})\subseteq D_1\cap J$. It is not hard to prove the reverse inclusion. Since $e_{11}-e_{nn}+r$ is a homogeneous element of degree $1$, we obtain that $D_1\cap J$ is a graded ideal. By Lemma \ref{maindivision}, the subspace $\mathcal{T}':=\mathcal{T}+J$ ($=\mathrm{Span}\{e_{11},e_{nn}\}\oplus J\oplus\mathbb{F}I$) is graded. Also,
\[
\mathcal{T}'/(J+\mathbb{F}\mathrm{I})\cong\mathcal{T}'\cap D_1/((D_1\cap J)+\mathbb{F}\mathrm{I})=D_1/((D_1\cap J)+\mathbb{F}\mathrm{I}).
\]
Thus, the image of $\mathrm{Span}\{e_{11},e_{nn}\}$ in $D_1/((D_1\cap J) +\mathbb{F}\mathrm{I})$ is graded, and of the same kind as in Lemma \ref{maindivision}. This proves the statement.
\end{proof}

As a consequence of the previous lemma, there exist $r$, $s\in J\cap D_1$ such that $e_{11}-e_{nn}+r$ and $e_{11}+e_{nn}+s$ are semihomogeneous. Moreover, $\deg (e_{11}-e_{nn}+r)=\left(\deg e_{11}+e_{nn}+s\right)^2=1$.

\begin{Lemma}
We may choose $r$ and $s$ such that $[e_{11}-e_{nn}+r,e_{11}+e_{nn}+s]=0$.
\end{Lemma}
\begin{proof}
For simplicity, let $x^-=e_{11}-e_{nn}+r$ and $x^+=e_{11}+e_{nn}+s$. Denote $s_0=s-(s)_{1n}e_{1n}$, $r_0=r-(r)_{1n}e_{1n}$, and $r_0^-=[e_{11}+e_{nn},r_0]$ (note that $r_0^-$ is $r_0$ where the rightmost column has inverted sign), $s_0^-=[e_{11}+e_{nn},s_0]$. Then we have
\begin{align*}
[x^-,x^+]&=s+(s)_{1n}e_{1n}-r_0^-+[r,s],\\{}%
[x^+,[x^-,x^+]]&=s_0^--r_0+[r_0^-,s],\\{}%
[x^+,[x^+,[x^-,x^+]]]&=s_0-r_0^-+[s,s_0^-]-[s,r_0].
\end{align*}
Thus, since $[x^+,[x^-,x^+]]$ is homogeneous of degree $1$, we may replace $r$ by $r+[x^+,[x^-,x^+]]$. This gives
\[
[x^-+[x^+,[x^-,x^+]],x^+]=2(s)_{1n}e_{1n}-[s,s_0^-]=:we_{1n},
\]
for some $w\in\mathbb{F}$. The element $w$ is either $0$ or non-zero, and the former implies $\deg e_{1n}=\deg x^+$. Thus in the last case, we replace $s$ by $s-\frac12we_{1n}$. Hence
\[
\left[x^-+[x^+,[x^-,x^+]],x^+-\frac12we_{1n}\right]=we_{1n}-we_{1n}=0.
\]
\end{proof}

\begin{Lemma}\label{step_type2}
Every $G$-grading on $\mathrm{UT}_n^{(-)}$ is isomorphic to one where $e_{11}-e_{nn}$ and $e_{11}+e_{nn}$ are semihomogeneous, and $\deg e_{11}-e_{nn}=\left(\deg e_{11}+e_{nn}\right)^2=1$.
\end{Lemma}
\begin{proof}
In the notation of the previous lemma, it is straightforward to check that the minimal polynomial of $e_{11}-e_{nn}+r$ is $X(X-1)(X+1)$, and the minimal polynomial of $e_{11}+e_{11}+s$ is $X(X-1)$. Thus both elements are diagonalizable, and since they commute, they are simultaneously diagonalizable. The diagonalization gives a new $G$-grading on $\mathrm{UT}_n^{(-)}$ where $e_{11}-e_{nn}$ and $e_{11}+e_{nn}$ are semihomogeneous elements satisfying the required properties on their degrees.
\end{proof}

Recall that we denote $X_{ij}^\pm=e_{ij}\pm e_{n-j+1,n-i+1}$.
\begin{Lemma}\label{existence_semihomogeneous2}
Up to an isomorphism, we may assume the elements $X_{ii}^\pm$ semihomogeneous, for $i=1$, 2, \dots, $n$. Moreover,
\begin{align*}
\deg X_{11}^- &=\deg X_{22}^-=\cdots=\deg X_{nn}^-=1,\\
\deg X_{11}^+&=\deg X_{22}^+=\cdots=\deg X_{nn}^+=g,
\end{align*}
where $g^2=1$ (it may happen that $g=1$).
\end{Lemma}
\begin{proof}
If $n=3$ then Lemma \ref{step_type2} states that $X_{11}^\pm$ are semihomogeneous. Thus $X_{22}^+=e_{22}=\mathrm{I}-X_{11}^+$ is semihomogeneous as well and $\deg X_{11}^+=\deg X_{22}^+$. So we may assume $n>3$. By the previous lemma, we may assume that $X_{11}^+$ and $X_{11}^-$ are homogeneous elements and $\deg X_{11}^+=g$, where $g^2=1$. Let
\[
\mathcal{U}_1=\mathrm{C}_{\mathrm{UT}_n^{(-)}}(J^{n-1})\bigcap\left(\left(1-(1/2)\mathrm{ad}(X_{11}^-)\right)\circ\left(1-\mathrm{ad}(X_{11}^-)\right)\left(\mathrm{UT}_n^{(-)}\right)\right).
\]
Here, and in what follows, $f\circ g$ stands for the composition of the functions $f$ and $g$ (first applying $g$ and then $f$). 
Then one can check that $\mathcal{U}_1$ is a graded subalgebra, and it is  isomorphic to $\mathrm{UT}_{n-2}^{(-)}\oplus\mathbb{F}(e_{11}+e_{nn})$. Then, repeating Lemma \ref{step_type2} to $\mathcal{U}_1$ (or, using Lemma \ref{ut2} if $n=4$), we see that there exist $\lambda^+$, $\lambda^-\in\mathbb{F}$ such that $\lambda^+(e_{11}+e_{nn})+X_{22}^+$ and $\lambda^-(e_{11}+e_{nn})+X_{22}^-$ are semihomogeneous in $\mathrm{UT}_n^{(-)}$, of degree $g_2$, and $1$, respectively. Our goal is to prove that $g_2=g$ and $\lambda^-=0$. This will imply that $X_{22}^+$ is semihomogeneous of degree $g$ and $X_{22}^-$ is semihomogeneous of degree $1$. Then an induction will complete the proof of the lemma.

If $\lambda^+\ne0$, then we may pick a nonzero homogeneous element in $J^{n-2}/J^{n-1}$. Note that
\[
[X_{11}^+,x]=\frac1{\lambda^+}[\lambda^+(e_{11}+e_{nn})+X_{22}^+]\ne0,
\]
thus $g_2=g$.

If $\lambda^+=0$ then consider $\mathcal{V}:=([X_{11}^-,J]+J^2)/J^2$. Therefore, for a nonzero homogeneous $x\in\mathcal{V}$, we have
\[
[X_{11}^+,x]=-[X_{22}^+,x]\ne0,
\]
and hence, once again we get $g_2=g$.

Finally, if $\lambda^-\ne0$, let $x\in J^{n-2}/J^{n-1}$ be a nonzero homogeneous element. Then
\[
[X_{11}^+,x]=\frac1{\lambda^-}[\lambda^-(e_{11}+e_{nn})+X_{22}^-]\ne0
\]
would imply $g=\deg X_{11}^+=\deg(\lambda^-(e_{11}+e_{nn})+X_{22}^-)=1$. Thus, a linear combination of these elements shows that $X_{22}^-$ is semihomogenous of degree $1$. The proof is complete.
\end{proof}

In the previous lemma, if $g=1$ then a linear combination of $X_{ii}^+$ and $X_{ii}^-$, for each $i=1$, \dots, $n$, shows that $e_{jj}$, for $j=1$, \dots, $n$, are semihomogeneous elements in the grading. Thus, Lemma \ref{semihomogeneous_elem} applies and the grading is a practical elementary one. Alternatively, Lemma \ref{existence_semihomogeneous} could be applied, since $g=1$. On the other hand, if $g\ne1$, then the grading is a practical type 2 one:

\begin{Lemma}\label{semihomogeneous_type2}
In the notation of the previous lemma, assume that $g\ne1$ (that is, $g\ne1$ in Lemma \ref{maindivision}). Then the grading is a practical type 2 grading.
\end{Lemma}
\begin{proof}
It is enough to prove that, for $i=1$, 2, \dots, $\lceil\frac{n-1}2\rceil$, the subspace $W_i=\mathrm{Span}\{X_{i,i+1}^\pm\}$ is graded. For, unless $n$ is even and $i=n/2$, then $\dim W_i=2$. Since $X_{ii}^+$ is semihomogeneous of degree $g\ne1$, we obtain that a basis of $W_i$ consisting of homogeneous elements may be chosen as $e_{i,i+1}\pm\epsilon_ie_{n-i,n-i+1}$. If $n$ is even and $i=n/2$ then $\dim W_i=1$ and there is nothing to do. Thus, up to the degree of the identity matrix, this is exactly the situation as in the construction \eqref{type2grading2}. So, repeating the construction of the matrix after \eqref{type2grading2}, we obtain that the grading is a practical type 2 grading.

Now let $i$ be such that $1\le i<\lceil\frac{n-1}2\rceil$. Then
\[
\left(1-\mathrm{ad}\,X_{i+1,i+1}^ -\right)\circ\mathrm{ad}\left(X_{ii}^-\right)\circ\mathrm{ad}\left(X_{i+1,i+1}^-\right) \left(\mathrm{UT}_n^{(-)}\right)=W_i
\]
is graded. If $n=2m$, then $\lceil\frac{n-1}2\rceil=m$ and
$$
\mathrm{ad}(X_{m,m}^-)\circ\mathrm{ad}(X_{m,m}^-)(J)=\mathrm{Span}\{e_{m,m+1}\}=W_m.
$$
If $n=2m+1$, then $\lceil\frac{n-1}2\rceil=m$ and
$$
\mathrm{ad}\left(X_{m+1,m+1}^+\right)\circ\mathrm{ad}\left(X_{m,m}^-\right)\left(UT_n^{(-)}\right)=\mathrm{Span}\{e_{m,m+1},e_{m+1,m+2}\}=W_m.
$$
The proof is complete.
\end{proof}

\subsection{} Now we have the ingredients to prove our first main result:
\begin{proof}[Proof of Theorem \ref{thm1}]
If $n=2$ then we apply Lemma \ref{ut2} and Lemma \ref{fact}. If $n>2$, then we consider the element $g$ of Lemma \ref{maindivision}. If $g=1$ then  Lemma \ref{existence_semihomogeneous} and Lemma \ref{semihomogeneous_elem} give us the first statement of (1). If $g\ne1$, then Lemma \ref{semihomogeneous_type2} proves that we get the first statement of (2). Now, Lemma \ref{fact} establishes the remaining statements of (1) and (2). Finally, since either (1) or (2) holds, the commutativity of the essential support is proved in Lemmas \ref{ess_supp} and \ref{ess_supp2}. If $\mathrm{char}\,\mathbb{F}=2$, then Lemma \ref{division_char2} proves that (1) remains valid, while (2) does not.
\end{proof}

\subsection{} Given a sequence $\eta=(g_1,\ldots,g_{n-1})\in G^{n-1}$, recall that
\[
\mathrm{rev}\,\eta=(g_{n-1},\ldots,g_1).
\]
As before we denote $\eta\sim\eta'$ if $\eta=\eta'$ or $\eta=\mathrm{rev}\,\eta'$.

If $\Gamma$ is a $G$-grading on $\mathrm{UT}_n^{(-)}$, denote by $T_G\bigl(\mathrm{UT}_n^{(-)},\Gamma\bigr)$, or $T_G(\Gamma)$ or simply by $T_G(\mathrm{UT}_n^{(-)})$ (when the grading can be inferred from the context) its $T_G$-ideal of $G$-graded polynomial identities. Since the identity matrix $\mathrm{I}$ annihilates any Lie bracket, it is easy to verify the following fact:
\begin{Lemma}\label{same_id_quotient}\
\begin{enumerate}
\item Consider the induced grading on $\mathcal{U}=\mathrm{UT}_n^{(-)}/\mathfrak{z}(\mathrm{UT}_n^{(-)})$ by the practically elementary grading $\Gamma_1(t,\eta)$. Then $T_G(\mathcal{U})=T_G(\mathrm{UT}_n^{(-)},\Gamma_1(\eta))$.
\item Consider the induced grading on $\mathcal{U}=\mathrm{UT}_n^{(-)}/\mathfrak{z}(\mathrm{UT}_n^{(-)})$ by the practical type 2 grading $\Gamma_2(t,g,\eta)$. Then $T_G(\mathcal{U})=T_G(\mathrm{UT}_n^{(-)},\Gamma_2(g,\eta))$.\qed
\end{enumerate}
\end{Lemma}

Using \cite[Theorem 1]{HY2020}, we may repeat the proof of \cite[Corollary 10]{KY2017a}, obtaining
\begin{Lemma}\label{id_of_one_not_other}
Let $\eta=(g_1,\ldots,g_{n-1})$ and $\eta'\in G^{n-1}$ be such that $\eta\not\sim\eta'$. For each $i=1$, 2, \dots, $n-1$, let $\xi_i=[x_{2i-1}^{(1)},x_{2i}^{(g_i)}]$. Then, up to possibly switching $\eta$ and $\eta'$, there exists a permutation $\sigma\in\mathcal{S}_{n-1}$ such that the polynomial
\[
\xi_\sigma:=[\xi_{\sigma(1)},\xi_{\sigma(2)},\ldots,\xi_{\sigma(n-1)}]
\]
belongs to $T_G(\Gamma_1(\eta))\setminus T_G(\Gamma_1(\eta'))$.\qed
\end{Lemma}
It is immediate that the automorphism $\omega$ of $\mathrm{UT}_n^{(-)}$ is a $G$-graded isomorphism between the $G$-gradings defined by $(t,\eta)$ and $(t,\mathrm{rev}\,\eta)$, for every $t\in G$. Combining this with the previous lemma, as an immediate consequence, we can state the following:
\begin{Lemma}\label{same_id}
Let $\eta$, $\eta'\in G^{n-1}$ be two sequences of pairwise commuting elements. Then $T_G(\Gamma_1(\eta))=T_G(\Gamma_1(\eta'))$ if and only if $\eta\sim\eta'$, and if and only if $\Gamma_1(\eta)\cong\Gamma_1(\eta')$.
\end{Lemma}
\begin{proof}
As noted before, $\eta\sim\eta'$ implies $\Gamma_1(\eta)\cong\Gamma_1(\eta')$. The former one readily implies $T_G(\Gamma_1(\eta))=T_G(\Gamma_1(\eta'))$. Now, if $\eta\not\sim\eta'$, then the previous lemma implies that $T_G(\Gamma_1(\eta))\ne T_G(\Gamma_1(\eta'))$, proving all the equivalences of this lemma.
\end{proof}

Let $\eta=(g_1,\ldots,g_{n-1})$, $\eta'=(g_1',\ldots,g_{n-1}')\in G^{n-1}$, and $g\in G$, $g^2=1$. Recall that $\eta\sim_g\eta'$ denotes that 

(a) $g_i=g_i'$ or $g_i=gg_i'$ for $i=1$, \dots, $\lfloor\frac{n-1}2\rfloor$, $\lceil\frac{n-1}2\rceil+1$, \dots, $n-1$, 

(b) if $n-1$ is odd, then $g_{n/2}=g_{n/2}'$. 

In a similar fashion, we have the following:
\begin{Lemma}\label{same_id2}
If $n>2$ then no type 2 grading satisfies the same graded polynomial identities as an elementary grading. Moreover, $T_G(\Gamma_2(g,\eta))=T_G(\Gamma_2(g',\eta'))$ if and only if $g=g'$ and $\eta\sim_g\eta'$, if and only if $\Gamma_2(g,\eta)\cong\Gamma_2(g',\eta')$.
\end{Lemma}
\begin{proof}
Assume that $n>2$. Let $\Gamma_2(g,\eta)$ be a type 2 grading, recall that $J=[\mathrm{UT}_n^{(-)},\mathrm{UT}_n^{(-)}]$, and let $h\in\mathrm{Supp}\,J/J^2$ (that is, the support of the induced grading on the quotient $J/J^2$). Then we may choose $h$ such that the polynomial $f_{g,h}(x^g,x^h)=\mathrm{ad}(x^g)^nx^h$ is not a $G$-graded polynomial identity for $\Gamma_2(g,\eta)$.
Here the upper index $g$ in $x^g$ means $x$ is a variable of $G$-degree $g$. On the other hand, consider an elementary grading $\Gamma_1(\eta')$. Since every homogeneous element of non-trivial degree with respect to $\Gamma_1(\eta')$ belongs to the nilpotent radical of $\mathrm{UT}_n^{(-)}$, we see that $f_{g,h}$ is a $G$-graded polynomial identity for $\Gamma_1(\eta')$. This proves the first part of the lemma.

If $g=g'$ and $\eta\sim_g\eta'$, then Lemma \ref{iso_type2_grading} guarantees us that $\Gamma_2(g,\eta)\cong\Gamma_2(g',\eta')$. This implies $T_G(\Gamma_2(g,\eta))=T_G(\Gamma_2(g',\eta'))$. So suppose that either $g\ne g'$ or $\eta\not\sim_g\eta'$. In the first case, since every homogeneous element of $\Gamma_2(g',\eta')$-degree $g$ is in the nilpotent radical, the polynomial $f_{g,h}$ above is a $G$-graded polynomial identity for $\Gamma_2(g',\eta')$ (but not for $\Gamma_2(g,\eta)$, as deduced in the first part of the proof). 

So assume that $g=g'$ but $\eta\not\sim_g\eta'$. Denote $\eta=(g_1,\ldots,g_{n-1})$ and $\eta'=(g_1',\ldots,g_{n-1}')$.

Now we construct a graded polynomial satisfied by one grading but not by the other. The case when $n$ is odd is easier, and we consider it first.

So assume $n$ is an odd integer. Denote by $\bar{\eta}=(g_1\langle g\rangle,\ldots,g_{n-1}\langle g\rangle)\in\left(G/\langle g\rangle\right)^{n-1}$, and use a similar notation for $\bar{\eta}'$. The condition $\eta\not\sim_g\eta'$ implies that $\bar{\eta}\not\sim\bar{\eta}'$. Thus, up to switching $\eta$ and $\eta'$, by Lemma \ref{id_of_one_not_other}, there exists a permutation $\sigma\in\mathcal{S}_{n-1}$ and a $G/\langle g\rangle$-graded polynomial $\bar{\xi}_\sigma\in T_G(\bar{\eta})\setminus T_G(\bar{\eta}')$. For each $i=1$, 2, \dots, $n-1$, let $\xi_i'=[x_{2i-1}^{(1)}+x_{2i-1}^{(g)},x_{2i}^{(g_i)}+x_{2i}^{(gg_i)}]$. Then we set
\[
\xi'_\sigma:=[\xi'_{\sigma(1)},\xi'_{\sigma(2)},\ldots,\xi'_{\sigma(n-1)}].
\]
By construction, the $G$-graded polynomial $\xi_\sigma'$ behaves like the $(G/\langle g\rangle)$-graded polynomial $\bar{\xi}_\sigma$ on the induced grading $\Gamma_2(g,\eta)$, by the projection $G\to G/\langle g\rangle$. The induced grading of $\Gamma_2(g,\eta)$, by that projection, coincides with the elementary grading $\Gamma_1(\bar{\eta})$. Therefore $\xi_\sigma'\in T_G(\Gamma_2(g,\eta))\setminus T_G(\Gamma_2(g,\eta'))$.

Now suppose $n$ is even. Then a similar argument holds, with the exception that we cannot pass to the quotient $G/\langle g\rangle$ because of the condition (b) of the definition of $\sim_g$. The argument we present here works also if $n$ is odd. First we replace $\eta'$ by a new $\eta''$, where $\eta'\sim_g\eta''$, by using the following rule: if $1\le i<n/2$, and $g_i'=gg_i$, then put $g_i''=g_i$. Otherwise, $g_i''=g_i'$, and if $n$ is even then $g_{n/2}''=g_{n/2}'$. Now let $\eta''=(g_1'',\ldots,g_{n-1}'')$ be symmetric so that $\eta''\sim_g\eta'$. Since $\eta\not\sim_g\eta''$, we have $\eta\not\sim\eta''$. Then Lemma \ref{id_of_one_not_other} implies the existence of $\xi_\sigma\in T_G(\Gamma_1(\eta))\setminus T_G(\Gamma_1(\eta'))$. For each $i=1$, 2, \dots, $n-1$, let
\[
\xi_i'=\left\{\begin{array}{ll}%
[x_{2i-1}^{(1)},x_{2i}^{(g_i)}],&\text{ if $n$ is even and $i=\frac{n}2$}\\{}%
[x_{2i-1}^{(1)},x_{2i}^{(g_i)}+x_{2i}^{(gg_i)}],&\text{ otherwise}%
\end{array}\right..
\]
The polynomial $\xi_\sigma'=[\xi_{\sigma(1)}',\ldots,\xi_{\sigma(n-1)}']$ is such that $\xi_\sigma'\in T_G(\Gamma_2(g,\eta))\setminus T_G(\Gamma_2(g,\eta'))$. The proof is complete.
\end{proof}

In this way we obtain the proof of our second main result:
\begin{proof}[Proof of Theorem \ref{thm2}] The equivalence of statements  (1) and (2) is given by Definition \ref{practical_iso}. The implication (2) $\Rightarrow$ (3) is always true. Now, assume that $T_G(\mathcal{U}_1/\mathcal{Z}_1)=T_G(\mathcal{U}_2/\mathcal{Z}_2)$. Then, combining the classification given by Theorem \ref{thm1} and Lemma \ref{same_id_quotient}, we may replace the gradings on $\mathcal{U}_1$ and $\mathcal{U}_2$ by new gradings, which are practical graded isomorphic to the original ones, and both are either elementary, or type 2 gradings. We shall conclude that these new gradings are isomorphic, so that the original ones are practical graded isomorphic. By Lemma \ref{same_id2}, we obtain that either $n=2$ or both gradings are simultaneously elementary or of type 2. If $n>2$ and both gradings are elementary, then we apply Lemma \ref{same_id}; and if both gradings are of type 2, then we apply Lemma \ref{same_id2}. If $n=2$, then Lemma \ref{ut2} says that both gradings may be assumed to be elementary. Thus, Lemma \ref{same_id} gives the isomorphism. This completes the proof of the implication $(3)\Rightarrow(1)$.
\end{proof}

\subsection{} Finally, we only need to put the pieces together in order to prove our last main result.

\begin{proof}[Proof of Theorem \ref{thm3}]
(2) Let $\Gamma$ and $\Gamma'$ be the gradings on $\mathrm{UT}_n^{(-)}/\mathbb{F}\mathrm{I}$ induced by $\Gamma_1(t,\eta)$ and $\Gamma_1(t',\eta')$, respectively. By Lemma \ref{same_id_quotient}\,(1), we have
\[
T_G(\Gamma)=T_G(\Gamma_1(\eta)),\quad\text{and}\quad T_G(\Gamma')=T_G(\Gamma_1(\eta')).
\]
Thus $T_G(\Gamma)=T_G(\Gamma')$ if and only if $T_G(\Gamma_1(\eta))=T_G(\Gamma_1(\eta'))$. Moreover, Lemma \ref{same_id} yields that $T_G(\Gamma_1(\eta))=T_G(\Gamma_1(\eta'))$ if and only if $\eta\sim\eta'$. Also, Theorem \ref{thm2} tells us that $\Gamma_1(t,\eta)$ and $\Gamma_1(t',\eta')$ are practically graded isomorphic if and only if $T_G(\Gamma)=T_G(\Gamma')$. Combining all these statements, we obtain the proof of (2).

(1) If $\Gamma_1(t,\eta)\cong\Gamma_1(t',\eta')$ then their centres coincide, so $t=t'$. Moreover, they are practically graded isomorphic, so (2) gives $\eta\sim\eta'$. Conversely, if $\eta\sim\eta'$, then (2) tells us that $\Gamma_1(t,\eta)$ and $\Gamma_1(t',\eta')$ are practically graded isomorphic. Thus, since $t=t'$, Lemma \ref{kythm} guarantees the graded isomorphism.

Statements (3) and (4) are proved in a similar way.
\end{proof}

\end{document}